\documentclass[11pt]{article}
\usepackage{epsfig}
\usepackage{amssymb}
\usepackage{amsmath,ulem}

\usepackage{graphicx}

\usepackage{setspace}%this stops lists having so much gap in them
\usepackage{enumitem}
\setlist{nolistsep}

\usepackage{lscape}

\newcommand{\Label}{\label}

\newcommand{\PG}{\textup{PG}}

\newcommand{\Q}{\mathcal{Q}}
\newcommand{\T}{\mathcal{T}}
\renewcommand{\H}{\mathcal{H}}
\renewcommand{\P}{\mathcal{P}}

\newcommand{\K}{\mathcal{K}}

\newcommand{\M}{\mathcal{M}}
\renewcommand{\O}{\mathcal{O}}

\newcommand{\h}{{\mathsf h}}

\newcommand{\sw}{s}
\renewcommand{\sb}{t}

\newtheorem{theorem}{Theorem}[section]

\newtheorem{lemma}[theorem]{Lemma}

\newtheorem{remark}[theorem]{Remark}
\newtheorem{proposition}[theorem]{Proposition}

\newenvironment{proof}{\noindent{\bfseries Proof}\hspace{0.5em}}{ \null \hfill $\square$ \par}

\usepackage{fancyhdr}
\usepackage[usenames,dvipsnames]{color}

\begin{document}
%%%Fancy header usage
%\pagestyle{fancy}
%\fancyhf{}
%\fancyfoot[C]{\thepage}
%\lfoot{\color{Aquamarine}\jobname}
%\rfoot{\color{Aquamarine}\today}
%\lhead{\color{Aquamarine}\jobname{\rm.tex}}
%\rhead{\color{Aquamarine}\today}

\title{Characterising elliptic and hyperbolic hyperplanes\\ of the parabolic quadric $\Q(2n,q)$}
\author{Jeroen Schillewaert\thanks{JS is supported by a University of Auckland Faculty Development Research Fund} \and Geertrui Van de Voorde\thanks{This author is supported by the Marsden Fund Council administered by the Royal Society of New Zealand.}
\date{}
}

\maketitle

\begin{abstract}
We provide a natural characterisation for the sets of elliptic and hyperbolic hyperplanes of the parabolic quadric $\Q(2n,q)$ when $q$ is even. This characterisation is based on the number of elements of these sets through points and codimension $2$ spaces and generalises \cite{B-H-J,B-H-J-S}.

%The second one, for odd $q$, is based on the number of elements of these sets through points and lines.}
%Let $\H$ be a non-empty set of hyperplanes in $\PG(2n,q)$, $q$ even,  such that every point of $\PG(2n,q)$ lies in either $0$, $\frac12q^{2n-1}$ or $\frac12q^{n}(q^n+1)$ hyperplanes of $\H$, and a codimension 2 space that is contained in a hyperplane of $\H$ lies in at least $\frac12q$ hyperplanes of $\H$. Then $\H$ is the set of all hyperplanes which meet a given non-singular quadric $Q(2n,q)$ in a hyperbolic quadric.
\end{abstract}

AMS code: 51E20\\
Keywords: projective geometry, quadrics, hyperplanes

%\tableofcontents

\section{Introduction}

The characterisation of polar spaces by means of combinatorial properties can be traced back to the 1950's, with seminal work done by Segre and Tallini and their schools. This continued throughout the 1970s and 1980s by works from Beutelspacher, Thas, and their collaborators and has enjoyed a revival by means of a variety of strong results in the past two decades, see e.g. \cite{survey} and the references therein.
  
Recently,  \cite{B-H-J,B-H-J-S} provided characterisations of certain hyperplanes of quadrics. Our present work generalises the latter.

A hyperplane of $\PG(2n,q)$ meets a non-singular parabolic quadric $\Q(2n,q)$ either in a non-singular elliptic quadric $\Q^-(2n-1,q)$, a non-singular hyperbolic quadric $\Q^+(2n-1,q)$, or a cone, denoted by $P\Q(2n-2,q)$, with vertex a point $P$ and base a non-singular parabolic quadric $\Q(2n-2,q)$. We call these hyperplanes {\em elliptic}, {\em hyperbolic}, or {\em singular} hyperplanes respectively. In this paper, we provide a characterisation result for hyperbolic and elliptic hyperplanes.

\begin{remark}
Throughout the article we treat the similar elliptic and hyperbolic case simultaneously. To achieve this when making statements and carrying out computations we make the following convention: we will use the $\pm$ and the $\mp$ symbol
where $\pm$ reads as $+$ when we are in the hyperbolic case and $-$ in the elliptic case and vice versa for $\mp$. All statements and their proofs should be read by choosing either the top or bottom row of the symbols $\pm$ and $\mp$ consistently (and as such every Theorem containing the symbol $\pm$ proves two different Theorems, one for the hyperbolic case, and one for the elliptic case.)
\end{remark}

\begin{theorem}\label{th:main} Let $\H^\pm$ be a non-empty set of hyperplanes in $\PG(2n,q)$, $n\geq 2$, $q$ even, $q>2$
, or $\PG(4,2)$, such that
\begin{itemize}
\item[(I)] every point of $\PG(2n,q)$ lies in either $0$, $\frac12q^{2n-1}$ or $\frac12q^{n}(q^{n-1}\pm 1)$ hyperplanes of $\H^\pm$, and
\item[(II)] every codimension $2$ space that is contained in a hyperplane of $\H^\pm$ lies in at least $\frac12q$ hyperplanes of $\H^\pm$.
%every plane of $\PG(4,q)$ lies in either $0$, $\frac12q$ or $q$ solids of $\H$.
\end{itemize}

Then the set of points contained in $\frac12q^{n}(q^{n-1}{\pm} 1)$ hyperplanes of $\H^\pm$ forms a non-singular quadric $\Q=\Q(2n,q)$ and
$\H^\pm$ is the set of hyperplanes that meets $\Q$ in a non-singular quadric $\Q^\pm(2n-1,q)$ or $n=2$ and $\H^-$ is  the set of solids of $\PG(4,q)$ disjoint from a hyperoval.
\end{theorem}

\begin{remark}
Theorem \ref{th:main} simultaneously generalises the main results of \cite{B-H-J,B-H-J-S} from dimension $4$ to all even dimensions. Furthermore, in \cite{B-H-J}, Condition (II) states that every every plane in $\PG(4,q)$ has to lie on $0,\frac12q,$ or $q$ solids, so our Condition (II) is slightly less stringent.
\end{remark}

\begin{remark}[Independence of the axioms]\label{axioms}
For an easy counterexample satisfying (II) but not (I) let $\mathcal{H}$ be the set of hyperplanes not containing a given codimension $2$ space.
In Lemma \ref{sanitycheck} we will prove that a {\em quasi-quadric} as defined below satisfies (I). Hence the set of hyperplanes meeting a quasi-quadric which is not a quadric in $|\Q^\pm(2n-1,q)|$ points satisfies (I) but not (II).  
\end{remark}

A {\em parabolic quasi-quadric $\K$} in $\PG(2n,q)$, $q$ even, with nucleus $N$ (as defined in \cite[Example 3.1]{QQ})  is a set of $\frac{q^{2n}-1}{q-1}$ points 
such that each line through $N$ contains a unique point of $\K$ and every hyperplane not through $N$ meets $\K$ in either $\frac{(q^n+1)(q^{n-1}-1)}{q-1}$ or $\frac{(q^n-1)(q^{n-1}+1)}{q-1}$ points.

\begin{remark}
It is clear that every non-singular parabolic quadric is a parabolic quasi-quadric but it was shown in \cite[Theorem 10]{QQ} that not all parabolic quasi-quadrics with nucleus $N$ are isomorphic to $\Q(2n,q)$ if $n\geq 2$, $q>2$. 
We are currently investigating the construction of quasi-quadrics and quasi-Hermitian varieties and the existence of parabolic quasi-quadrics without nucleus in $\PG(4,2)$.
\end{remark}

As a first step towards the proof of Theorem~\ref{th:main}, we will show the following. 

\begin{proposition}\label{pr:main} Let $\H^\pm$ be a non-empty set of hyperplanes in $\PG(2n,q)$, $n\geq 2$, $q$ even, $q>2$, or $\PG(4,2)$, such that
 \begin{itemize}
\item[(I)] every point of $\PG(2n,q)$ lies in either $0$, $\frac12q^{2n-1}$ or $\frac12q^{n}(q^{n-1}\pm 1)$ hyperplanes of $\H^\pm$,
%every plane of $\PG(4,q)$ lies in either $0$, $\frac12q$ or $q$ solids of $\H$.
\end{itemize}

then the set of points contained in $\frac12q^{n}(q^{n-1}\pm 1)$ hyperplanes of $\H^\pm$ forms a parabolic quasi-quadric $\K$ and
$\H^\pm$ is the set of hyperplanes that meets $\K$ in $|\Q^\pm(2n-1,q)|$ points or or $n=2$ and $\H^-$ is  the set of solids of $\PG(4,q)$ disjoint from a hyperoval.
\end{proposition}

As indicated in Remark \ref{axioms}, the following lemma, together with the fact that not all quasi-quadrics are quadrics, shows that one needs to impose an extra condition on the hyperplanes in Proposition~\ref{pr:main} to be able to characterise them as those meeting a quadric.

The proof of this lemma is a straightforward generalisation of \cite[Lemma 1.7]{B-H-J-S}.

\begin{lemma} \label{sanitycheck} Let $\K$ be a parabolic quasi-quadric in $\PG(2n,q)$ with nucleus $N$ and let $\mathcal{H}^\pm$ be the set of hyperplanes meeting $\K$ in $\frac{(q^{n}\mp 1)(q^{n-1}\pm1)}{q-1}$ points. Then a point of $\PG(2n,q)$ lies on  $0$, $\frac12q^{n}(q^{n-1}\pm 1)$ or $\frac12q^{2n-1}$ hyperplanes of $\H^\pm$, i.e., $\H^\pm$ meets condition (I).
\end{lemma}
\begin{proof} Fix a point $X\in \K$ and let $H^\pm$ be the number of elements of $\H^\pm$ containing $X$. Note that by definition, $H^+ + H^-=q^{2n-1}$. We count in two ways the pairs $(P,\Pi)$ where $P\in \K$, $P\neq X$, and where $\Pi$ is a hyperplane containing $P,X$ but not $N$.  For the left hand side, we first count the hyperplanes through $X$ but not $N$, and then consider the points in those hyperplanes.
 For the right hand side we first look at the points $P\in \K$, $P\neq X$. Note that from the definition of the quasi-quadric the line $PX$ does not contain $N$, so there are $q^{2n-2}$ hyperplanes
through $PX$ which do not contain $N$. We find
$$H^\pm\left(\frac{(q^{n}\mp1)(q^{n-1}\pm 1)}{q-1}-1\right)+\left(q^{2n-1}-H^\pm\right)\left(\frac{(q^{n}\pm 1)(q^{n-1}\mp1)}{q-1}-1\right)=\left(\frac{q^{2n}-1}{q-1}-1\right)q^{2n-2},$$
which implies that $H^\pm=\frac12q^{n}(q^{n-1}\pm 1)$.
We use the same argument for a point $Y$ not belonging to $\K$ and different from $N$. Let $T^\pm$ be the number of elements of $H^\pm$ containing $Y$. Note that the line $NY$ has one point $Z$ of $\K$ so there are no hyperplanes in $H^\pm$ containing $Y$ and $Z$.
$$T^\pm\left(\frac{(q^{n}\mp1)(q^{n-1}\pm 1)}{q-1}\right)+\left(q^{2n-1}-T^\pm\right)\left(\frac{(q^{n}\pm 1)(q^{n-1}\mp1)}{q-1}\right)=\left(\frac{q^{2n}-1}{q-1}-1\right)q^{2n-2}.$$

Hence $T^\pm=\frac12q^{2n-1}$.
\end{proof}

To find the desired characterisation, we add condition (II) as hypothesis which is the direct generalisation of the restriction used in \cite{B-H-J-S} and which is a less stringent condition than the one in \cite{B-H-J}. As opposed to \cite{B-H-J} we are also able to include the case $\PG(4,2)$, see \cite[Lemma 1.6]{B-H-J-S}.

\begin{lemma}\label{axiom-verification} The set $\H^\pm$ of hyperplanes meeting a non-singular parabolic quadric $\Q(2n,q)$, $q$ even, in $|\Q^\pm(2n-1,q)|$ points satisfies Conditions (I)-(II). Furthermore, the set $\H^-$ of solids disjoint from a hyperoval in a plane in $\PG(4,q)$, $q$ even, satisfies Conditions (I)-(II).
\end{lemma}
\begin{proof} Let $\H^\pm$ be the set of hyperplanes meeting a non-singular parabolic quadric $\Q(2n,q)$ in $|\Q^\pm(2n-1,q)|$ points. In view of Lemma \ref{sanitycheck} we only need to verify Condition (II).

 Let $\Q=\Q(2n,q)$ with nucleus $N$ and let $\pi$ be a codimension $2$ space contained in a hyperplane of $\H^\pm$. Since $\H^\pm\cap \Q=\Q^\pm(2n-1,q)$, we know that $\pi$ intersects $\Q$ either in a $\Q(2n-2,q)$ or a cone $p\Q^\pm(2n-3,q)$. The hyperplane $H$ spanned by $N$ and $\pi$ meets $\Q$ in $\frac{q^{2n-1}-1}{q-1}$ points, while all other hyperplanes through $\pi$ meet either in $|\Q^+(2n-1,q)|$ or $|\Q^-(2n-1,q)|$ points. 

Suppose that $\pi$ intersects $\Q$ in $\Q(2n-2,q)$. If $X$ is the number of hyperplanes of $\H^\pm$ through $\pi$, then counting the number of points of $\Q$ in hyperplanes through $\pi$ yields that
 $$\frac{q^{2n-1}-1}{q-1}+X(|\Q^\pm(2n-1,q)|-|\Q(2n-2,q)|)+(q-X)(|\Q^\mp(2n-1,q)|-|\Q(2n-2,q)|)=|\Q(2n,q)|,$$ and it follows that $X=\frac12q.$

If $\pi$ intersects $\Q$ in a cone $p\Q^\pm(2n-3,q)$, then no hyperplanes meeting $\Q$ in $\Q^-(2n-1,q)$ can go through $\pi$. So, apart from the hyperplane $H$ containing $N$, all other $q$ hyperplanes through $\pi$ belong to $\H^\pm$. We conclude that a codimension $2$ space lies on $0$, $\frac12q$ or $q$ hyperplanes of $\H^\pm$, and hence, $\H^\pm$ satisfies Condition (II).

%Using Theorem 1.57, Lemma 1.13 and Corollary 1.68 of \cite{HT} (see also Example 1.70), 
%we find that any codimension $2$ space not containing the nucleus which meets $\Q(2n,q)$ in a $\Q(2n-2,q)$ is contained in exactly $\frac q2$ hyperbolic hyperplanes, any codimension $2$ space meeting $\Q(2n,q)$ in $p\Q^+(2n-3,q)$ is contained in exactly $q$ hyperbolic hyperplanes, and all other codimension $2$ spaces are not contained in any hyperbolic hyperplanes.

Now let $\H^-$ be the set of solids in $\PG(4,q)$ disjoint from a hyperoval $H$ in a plane $\pi$. A point $P$ of $H$ lies on $0$ elements of $\H^-$, a point $Q$ of $\pi$, not in $H$, lies on $\frac q2$ external lines to $H$ in $\pi$, each lying on $q^2$ solids not containing $\pi$, which means $Q$ lies on $\frac12q^3$ solids of $\H^-$. Let $R$ be a point of $\PG(4,q)$ not in $\pi$. For a given line $\ell$ of $\pi$ there are exactly $q$ solids containing $\ell$ and $R$ but not $\pi$. Since there are $\frac12(q^2-q)$ external lines to $H$ in $\pi$ we obtain that $R$ lies on $\frac12(q^3-q^2)$ elements of $\H^-$.

Any plane containing at least one point of $H$ lies on $0$ solids of $\H^-$. A plane $\mu$ meeting $\pi$ in an external line lies on precisely $q$ solids of $\H^-$, namely those solids through $\mu$ different from $\langle \pi,\mu\rangle$. Now consider a plane $\nu$ meeting $\pi$ exactly in the point $S\notin H$. Each of the $\frac q2$ external lines to $H$ through $S$ gives rise to a solid of $\H^-$ and vice versa. So we find that $\nu$ lies on $\frac q2$ solids of $H^-$.
\end{proof}

To prove Theorem \ref{th:main}, we will use the following result by De Winter and Schillewaert.
\begin{theorem}\cite[Theorem 1.6]{DWS}\label{SDW-JS}
If a point set $\K$ in $\PG(n,q)$, $n\geq 4$, $q>2$, or $\PG(4,2)$ has the same intersection numbers with respect to hyperplanes and subspaces of codimension $2$ as a non-singular polar space $\P\in\{\H(n,q),\Q^+(n,q),\Q^-(n,q),\Q(n,q)\}$, then $\K$ is the point set of the non-singular polar space $\P$.
\end{theorem}

\section{Proof of the characterisation results}
\subsection{Proof of Proposition \ref{pr:main}}

Throughout this section, let $\H^{\pm}$ be a non-empty set of hyperplanes in $\PG(2n,q)$, $q$ even, $n\geq 2$, such that
 \begin{itemize}
\item[(I)] every point of $\PG(2n,q)$ lies on $0$ or $\frac12q^{2n-1}$ or $\frac12q^{n}(q^{n-1}\pm1)$ hyperplanes of $\H^{\pm}$.
%\item[(II)] every $(2n-2)$-space that is contained in a hyperplane of $\H^{\pm}$ lies in at least $\frac12q$ hyperplanes of $\H^{\pm}$.
\end{itemize}

%So throughout this section, let $\H^{\pm}$ be a non-empty set of hyperplanes in $\PG(2n,q)$, $q$ even,   such that
% \begin{itemize}
%\item[(I)] every point of $\PG(2n,q)$ lies on $0$ or $\frac12q^{2n-1}$ or $\frac12q^{n}(q^{n-1}\pm1)$ hyperplanes of $\H^{\pm}$, and 
%\item[(II)] every $(2n-2)$-space that is contained in a hyperplane of $\H^{\pm}$ lies in at least $\frac12q$ hyperplanes of $\H^{\pm}$.
%\end{itemize}

By assumption (I), there are three types of points. A point is called a
\begin{itemize}
\item \emph{red point} if it lies in $0$ hyperplanes of $\H^{\pm}$,
\item \emph{white point} if it lies in $\frac12q^{2n-1}$ hyperplanes of $\H^{\pm}$,
\item \emph{black point} if it lies in $\frac12q^{n}(q^{n-1}\pm 1)$ hyperplanes of $\H^{\pm}$.
\end{itemize}

Throughout, we will denote the number of red points of $\PG(2n,q)$ by $r$, the number of white points of $\PG(2n,q)$ by $w$, and the number of black points of $\PG(2n,q)$ by $b$.

Our aim is to show that the black points form a parabolic quasi-quadric with a unique red point as its
nucleus.

A hyperplane in $\H^{\pm}$ does not contain any red points, and we partition the hyperplanes of $\PG(2n,q)$ into $\H^{\pm}$ and two further types of hyperplanes:
\begin{itemize}
\item
let $\T$ be the set of hyperplanes containing at least one red point,
\item
let $\M^{\pm}$ be the set of hyperplanes not in $\H^{\pm}$ that contain no red points.
\end{itemize}

We proceed with a series of lemmas, based on the ideas of \cite{B-H-J,B-H-J-S}.

\begin{lemma}\Label{lem:A-size1}
Let  $$\h^\pm= \dfrac{|\H^{\pm}|}{\frac12q^{n}},$$  then $\h^\pm$ is an integer congruent to either $0$ or $\pm 1$ modulo $q^{n-1}$.
\end{lemma}

\begin{proof}
We first show that $\frac12q^{n}$ divides $|\H^{\pm}|$.
Count in two ways the incident pairs $(P,\Pi)$ where $\Pi$ is a hyperplane in $\H^{\pm}$ and $P$ is a point of $\Pi$:
\begin{eqnarray}
w\tfrac12 q^{2n-1} + b\tfrac12q^{n}(q^{n-1}\pm 1)& =& |\H^{\pm}|\frac{q^{2n}-1}{q-1}.\label{eqn:OP}
\end{eqnarray}
Since $q=2^h$ we have that $\gcd(\frac12q^{n},\frac{q^{2n}-1}{q-1})=1$. Moreover, $\frac12q^{n}$ divides the left hand side of \eqref{eqn:OP}. It follows that $\frac12q^{n}$ divides $|\H^{\pm}|$.

To show that $\h^\pm= \dfrac{|\H^{\pm}|}{\frac12q^{n}}$ is congruent to either $0$ or $\pm 1$ modulo $q^{n-1}$.
we count in two ways incident triples $(P , \Pi,\Sigma)$ where $\Pi,\Sigma$ are distinct hyperplanes of $\H^\pm$ and $P$ is a point in $\Pi\cap\Sigma$. We have
\begin{eqnarray}
\small
w\tfrac12q^{2n-1}(\tfrac12q^{2n-1}-1)+b\tfrac12 q^n(q^{n-1}\pm1)(\tfrac12 q^n(q^{n-1}\pm1)-1)&=&|\H^{\pm}|(|\H^{\pm}|-1)\frac{q^{2n-1}-1}{q-1}\label{eqn:1}
\end{eqnarray}
Substituting (\ref{eqn:OP}) into (\ref{eqn:1}) to eliminate $b$ and then solving for $w$ and using $|\H^{\pm}|=\frac12q^n\h^\pm$ gives
\begin{align}
w=\frac{\h^\pm(q^{2n}-1\pm (\h^\pm-\h^\pm q^{2n-1}-2q^n+q^{3n-1}+q^{n-1}))}{q^{n-1}(q-1)}\label{eqn:wq}
\end{align}
Thus the numerator must be divisible by $q^{n-1}$ and hence $\h^\pm (-1\pm\h^\pm)\equiv 0 \pmod{q^{n-1}}$. 
As $q$ is a power of 2, it follows that either $\h^\pm \equiv 0 \pmod{q^{n-1}}$ or $\h^\pm \equiv \pm 1 \pmod{q^{n-1}}$.
\end{proof}

\begin{lemma}\Label{item:hsize} %The following statements hold:
%\begin{itemize}
%\item Each solid in $\H^\pm$ contains $|Q^{\pm}(2n-1,q)|=\frac{(q^{n}\mp1)(q^{n-1}\pm1)}{q-1}$ black points
%\item $|\H^\pm|=\frac12 q^{n}(q^{n}\pm 1)$
%\item 

There are $|\Q(2n,q)|$ black points and one red point, or $n=2$, $|\H^-|=\frac12q^3(q-1)$ and there are $q+2$ red points, $q^2-1$ white points and $q^4+q^3$ black points.
%\end{itemize}
\end{lemma}

\begin{proof}
Let $\Sigma\in\H^\pm$, let $\sw $ be the number of white points in $\Sigma$ and let $\sb $ be the number of black points in $\Sigma$. As $\Sigma$ does not contain any red points, we have
\begin{eqnarray}
\sw +\sb &=&\frac{q^{2n}-1}{q-1}.\label{eqn:3-pt-A}
\end{eqnarray}
Count in two ways incident pairs $(P,\Pi)$ where $\Pi\in\H^\pm$, $\Sigma\ne \Pi$ and $P$ is a point in $\Sigma\cap\Pi$. Using the notation of Lemma \ref{lem:A-size1}, we have that
\begin{eqnarray*}
\sw (\tfrac12 q^{2n-1}-1)+\sb (\tfrac12q^{n}(q^{n-1}\pm 1)-1)&=&(\h^\pm\tfrac12q^n-1)\frac{q^{2n-1}-1}{q-1}.
%\\
%(\sw +\sb )(\tfrac12q^{2n-1}-1)+\sb \tfrac12 q^{2n-2}&=&(|\H|-1)(q^{2n-2}+q+1)
\end{eqnarray*}

Using \eqref{eqn:3-pt-A} to eliminate $s$ and then solving for $\h^\pm$
 yields
\begin{eqnarray}
\h^\pm&=&\frac{q^{3n-1}-2q^{n}+q^{n-1}\pm t (q-1)}{q^{2n-1}-1}.\label{eqn:3-pt-A-m}
\end{eqnarray}

Note that $t$ and hence $s$ are thus constants.

Since $ 0\leq t \leq \frac{q^{2n}-1}{q-1}$ %leq \tfrac{q^{2n}-1}{q-1}$, 
it follows that 
\begin{align}
%q^n-q^n\frac{1-q^{-1}}{q^{2n-1}-1}
q^n-1< \h^+< q^n+q, {\mathrm{ and}}\\
q^n-q-1<\h^-<q^n.
\end{align}

Since we have shown in Lemma \ref{lem:A-size1} that $\h^\pm$ is congruent to $0$ or $\pm 1 \pmod{q^{n-1}}$ it follows that $\h^+\in \{q^n,q^n+1\}$. We also see that if $n>2$, then $\h^-=q^n-1$ and if $n=2$, then $\h^-\in \{q^2-q,q^2-1\}$.
Using \eqref{eqn:3-pt-A-m}, we conclude that if $(|\H^\pm|,t)\neq (\tfrac12q^{n}(q^n\pm 1),|\Q^\pm(2n-1,q)|)%=(\tfrac12q^{n}(q^n+1),q^{n-1}+\frac{q^{2n-1}-1}{q-1}))
$, then $(|\H^+|,t)=(\tfrac12q^{2n},q^{n-1})$ or $n=2$ and $(|\H^-|,t)=(\frac12q^2(q^2-q),q^2(q+1))$.

Count in two ways incident pairs
 $(P,\Sigma)$ where $\Sigma\in\H^\pm$ and $P$ is black point in $\Sigma$:
\begin{eqnarray}
b\tfrac{q^n}{2}(q^{n-1}\pm1)= |\H^\pm|t. \label{eqb}
\end{eqnarray}

Count in two ways incident pairs
 $(P,\Sigma)$ where $\Sigma\in\H^\pm$ and $P$ is a white point in $\Sigma$:
\begin{eqnarray}
w\tfrac12q^{2n-1}= |\H^\pm|s. \label{eqw}
\end{eqnarray}

%
%By Lemma~\ref{item:hsize}, we have that either $(|\H^\pm|,t)=(\tfrac12q^{2n},q^{n-1})$
%or $(|\H^\pm|,t)=(\tfrac12q^{n}(q^n+1),q^{n-1}+\frac{q^{2n-1}-1}{q-1})$.
%

Now if $(|\H^+|,t)=(\tfrac12q^{2n},q^{n-1})$, then \eqref{eqb} yields $b=\tfrac{q^{2n-1}}{q^{n-1}+1}$, a contradiction since $b$ is an integer.
%We conclude that $|\H^\pm|=\tfrac12q^{n}(q^n+1)$, $t=|\Q^\pm(2n-1,q)|$ and $b=|\Q(2n,q)|$ as required, unless $n=2$ and $(|\H^-|,t)=(\frac12q^2(q^2-q),q^2(q+1))$.

Assume that $|\H^\pm|=\tfrac12q^{n}(q^n\pm 1)$ and $t=|\Q^\pm(2n-1,q)|$. It then follows from \eqref{eqb} that $b=|\Q(2n,q)|$, and from \eqref{eqw} that $w=\frac{q^{2n+1}-1}{q-1}-b-1$. This implies that there is a unique red point.

Now assume that $n=2$ and $(|\H^-|,t)=(\frac12q^2(q^2-q),q^2(q+1))$. It follows from \eqref{eqn:3-pt-A} that $s=q+1$, from \eqref{eqb} that $b=q^4+q^3$ and from \eqref{eqw} that $w=q^2-1$. Since $w+b+r=\frac{q^5-1}{q-1}$ we conclude that $r=q+2$.
\end{proof}

 The following proposition deals with the case $n=2$ and $(|\H^-|,t)=(\frac12q^2(q^2-q),q^2(q+1))$, using only Condition (I). We will show that the red points lie in a plane, essentially replacing Lemma 2.6 of \cite{B-H-J} which uses (their stronger version of) Condition (II). Note that Proposition \ref{hyperovalcase} also replaces Lemmas 2.7, 2.8, 2.9 of \cite{B-H-J}, but these could have been applied verbatim as they don't make use of Condition (II).

\begin{proposition}\label{hyperovalcase}If $n=2$ and $|\H^-|=\frac12q^2(q^2-q)$, then the $q+2$ red points form a hyperoval $\O$ in a plane and $\H^-$ is the set of solids skew from $\O$.
\end{proposition}
\begin{proof} Consider a line $L$ containing at least one red point. Note that every solid of $\H^-$ intersects $L$ in precisely one point, which is necessarily white or black.  Let $u$ be the number of white points on $L$ and $v$ be the number of black points on $L$. Counting the solids of $\H^-$ through points of $L$ in two ways, we have that
\begin{align}u\tfrac12q^3+v\tfrac12(q^3-q^2)=\tfrac12(q^4-q^3).\label{uv} \end{align}

 Hence $u+v\geq q-1$, so each line has at most $2$ red points.
Moreover, if $u+v=q-1$, it follows from \eqref{uv} that $v=0$. So every line through $2$ red points contains precisely $q-1$ white points.
If $L$ contains exactly one red point, equation of \eqref{uv} becomes
$$(q-v)q^3+vq^3-vq^2=q^4-q^3$$ which implies that $v=q$, so a tangent line through a red point contains precisely $q$ black points. Now consider $3$ red points $P_1,P_2,P_3$, which then necessarily span a plane $\pi$. The line $P_2P_3$ contains $q-1$ white points. Every line $M$ through $P_1$, different from $P_1P_2$ and $P_1P_3$ in the plane $\pi$ meets $P_2P_3$ in a white point. Since a tangent line through a red point 
does not contain white points, it follows that $M$ contains precisely one extra red point. Hence there are at least $q-1$ red points in $\pi$, different from $P_1,P_2,P_3$. Since there are $q+2$ red points, it follows all these $q+2$ points lie in $\pi$ and as every line contains at most $2$ of them, the red points form a hyperoval.
Recall that no solid of $\H^-$ contains a red point. Since the number of solids disjoint from a hyperoval is precisely $\frac12q^2(q^2-q)=|\H|$, every solid disjoint from $\O$ needs to be in $\H^-$, and the characterisation follows. 
\end{proof}

%
%
%Recall that we partition the hyperplanes of $\PG(2n,q)$ into three sets: $\H^\pm$, $\T$ (the hyperplanes containing the red point), and $\M^\pm$ (the hyperplanes not in $\H^\pm$ which do not contain the red point). 

{\bf Proof of Proposition \ref{pr:main}} 
%\begin{lemma}\Label{item:int} If not $n=2$ and $(|\H^-|,t)=(\frac12q^2(q^2-q),q^2(q+1))$, then
%\begin{enumerate}
%\item Each solid in $\M^\pm$ contains $|\Q^\mp(2n-1,q)|$ black points.
%\item Each solid in $\T$ contains $|p\Q(2n-2,q)|$ black points.
%\end{enumerate}
%\end{lemma}

As before, denote by $t$ the number of black points in an element of $\H^\pm$.
By Lemma~\ref{item:hsize}, we either have that $n=2$ and $(|\H^-|,t)=(\frac12q^2(q^2-q),q^2(q+1))$, or $|\H^\pm|=\tfrac12q^{n}(q^n\pm1)$ and $t=|\Q^\pm(2n-1,q)|$. 
By Proposition \ref{hyperovalcase} we may assume that $|\H^\pm|=\tfrac12q^{n}(q^n\pm1)$ and $t=|\Q^\pm(2n-1,q)|$. Lemma~\ref{item:hsize} shows that in this case, there is a unique red point, so we have $|\T|=\frac{q^{2n}-1}{q-1}$ (the number of hyperplanes through a point). 

Let $\Pi\in\M^\pm$, let $a$ be the number of white points in $\Pi$ and let $b$ be the number of black points in $\Pi$. As solids in $\M^\pm$ do not contain a red point, we have
\begin{eqnarray}
a +b &=&\frac{q^{2n}-1}{q-1}.\label{eqn:Ept}
\end{eqnarray}
Count in two ways incident pairs $(P,\Sigma)$ where $\Sigma\in\H^\pm$, and $P\in\Pi\cap\Sigma$. We have
\begin{eqnarray*}
a\tfrac12 q^{2n-1}+b \tfrac12 q^{n}(q^{n-1}\pm1)&=&|\H^\pm|\frac{q^{2n-1}-1}{q-1}.
\end{eqnarray*}
Simplifying using (\ref{eqn:Ept}) gives $b = |\Q^\mp(2n-1,q)|$.

Now let $\Pi\in\T$, let $c$ be the number of white points in $\Pi$ and let $d $ be the number of black points in $\Pi$. As hyperplanes in $\T$ contain the unique red point, we have
\begin{eqnarray}
c +d &=& \frac{q^{2n}-1}{q-1}-1.\label{eqn-bing}
\end{eqnarray}
 Count in two ways incident pairs $(P,\Sigma)$ where $\Sigma\in\H^\pm$, and $P\in\Pi\cap\Sigma$, noting that as $\Sigma\in\H^\pm$, it does not contain the red point. We have
 \begin{eqnarray*}
c \tfrac12 q^{2n-1}+d \tfrac12 q^{n}(q^{n-1}\pm 1)&=&|\H^\pm|\frac{q^{2n-1}-1}{q-1}.%\\
\end{eqnarray*}
Simplifying using (\ref{eqn-bing}) gives $d = |p\Q(2n-2,q)|$. Let $\mathcal{B}$ be the set of black points and let $N$ denote the red point. We have shown that $|\mathcal{B}|=\frac{q^{2n}-1}{q-1}$ and that every hyperplane not through $N$ meets $\mathcal{B}$ in $ |\Q^\pm(2n-1,q)|$ points, and every hyperplane through $N$ meets $\mathcal{B}$ in $|p\Q(2n-2,q)|$ points.
Finally, we need to show that every line through $N$ contains precisely one point of $\mathcal{B}$. Consider a line $L$ containing $N$. Note that every hyperplane of $\H^\pm$ then intersects $L$ in precisely one point. Let $v$ be the number of black points on $L$ and $q-v$ be the number of white points on $L$. Counting the hyperplanes of $\H^\pm$ through points of $L$ in two ways, we have that
\begin{align}v\tfrac12q^n(q^{n-1}\pm 1)+(q-v)\tfrac12q^{2n-1}=\tfrac12q^{n}(q^n\pm1), \end{align}
which indeed implies that $v=1$.
Hence, $\mathcal{B}$ is a quasi-quadric with the red point $N$ as nucleus.
$\hfill\square$

%\begin{remark}\label{rem}
%Note that so far we have only used Condition (I) and we have proved that the set of black points is a quasi-quadric if not $n=2$ and $(|\H^-|,t)=(\frac12q^2(q^2-q),q^2(q+1))$.
%\end{remark}
\subsection{Proof of Theorem \ref{th:main}}

Assume now that $\H^{\pm}$ is a non-empty set of hyperplanes in $\PG(2n,q)$, $n\geq 2$, $q$ even, $q>2$
, or $\PG(4,2)$, such that
\begin{itemize}
\item[(I)] every point of $\PG(2n,q)$ lies on $0$ or $\frac12q^{2n-1}$ or $\frac12q^{n}(q^{n-1}\pm1)$ hyperplanes of $\H^{\pm}$, and 
\item[(II)] every codimension $2$ space contained in an element of $\H^{\pm}$ belongs to at least $\frac12q$ of them.
\end{itemize}

%\begin{lemma}\label{triviaalO} Let $\S$ be the set of solids in $\PG(4,q)$ disjoint from a hyperoval $\mathcal{O}$. Then $\S$ satisfies conditions (I) and (II).
%\end{lemma}
%\begin{proof} Let $\O$ be a hyperoval contained in a plane $\pi$. A point of $\O$ lies on no solids of $\S$. A point of $\pi$ lies on $\tfrac q2$ external lines to $O$ in $\pi$ which each lie on precisely $q^2$ solids not containing $\pi$. A point of $\PG(4,q)$ not in $\pi$ lies together with each of the $\frac{q^2-q}{2}$ external lines to $\O$ in $\pi$ on precisely $q$ solids not containing $\pi$. This shows that a point lies on $0$, $\frac{ q^3}{ 2}$ or $\frac {q^2}{ 2} (q-1)$ solids of $\S$, so $\S$ satisfies Condition (I).
%
%Now let $\mu$ be a plane contained in an element of $\S$, then $\mu$ meets $\pi$ either in an external point, say $R$ to $\O$ or an external line to $O$. In the latter case, we see that the $q$ solids through $\mu$ distinct from the solid $\langle \pi,\mu\rangle$ are in $\S$. In the former case, we find that the $q/2$ solids spanned by $\pi$ and one of the $q/2$ external lines through $R$ to $\O$ in $\pi$ are the precisely the solids of $\S$ through $\pi$. We conclude that $\pi$ lies in at least $q/2$ solids of $\S$, hence, $\S$ satisfies Condition (II).
%\end{proof}

%We continue using the same notations as before, but now for a set $\H^\pm$, satisfying Condition (II) on top of (I).
\begin{lemma}\label{plane-count-new} If there are $|\Q(2n,q)|$ black points, then the number of black points in a codimension $2$ space is one of $|\Q(2n-2,q)|$, $|p\Q^{+}(2n-3,q)|$ or $|p\Q^{-}(2n-3,q)|$.
\end{lemma}

\begin{proof} By Proposition \ref{pr:main}, the black points $\mathcal{B}$ form a parabolic quasi-quadric with nucleus the red point. It follows that a hyperplane in $\H^\pm$ meets $\mathcal{B}$ in $|\Q^\pm(2n-1,q)|$ points, a hyperplane in $\T$ meets $\mathcal{B}$ in $|p\Q(2n-2,q)|$ points, and a hyperplane in $\M^\pm$ meets $\mathcal{B}$ in $|\Q^\mp(2n-1,q)|$ points.

Let $\pi$ be a codimension $2$ space, and denote the number of black points in $\pi$ by $x^\pm$.  Let $t$ be the number of hyperplanes of $\T $ containing $\pi$ and let $s^{\pm}$ be the number of hyperplanes of $\H^\pm$  containing $\pi$, so there are $q+1-s^\pm-t$ hyperplanes of $\M^\pm$  containing $\pi$. 
Expressing that the total number of black points in $\PG(2n,q)$ can be obtained by grouping them in hyperplanes  containing $\pi$ yields $$(q+1-s^\pm-t)(|\Q^\mp(2n-1,q)|-x^\pm)+s^\pm(|\Q^\pm(2n-1,q)|-x^\pm)+t(|p\Q(2n-2,q)|-x^\pm) + x^\pm = |\Q(2n,q)|. $$

 If $\pi$ contains the red point, then $t=q+1$ and $s^\pm=0$, hence $x^\pm=|\Q(2n-2,q)|$ and we are done. So assume that $\pi$ does not contain the red point. Since the red point is contained in precisely one hyperplane through $\pi$, we have that $t=1$, and it follows that
\begin{align}
 x^\pm=2s^\pm(\pm q^{n-2})+\frac{q^{2n-2}\mp q^n\pm q^{n-1}-1}{q-1} \label{x-s}.
\end{align}
 
Now if $s^\pm = 0$, then $x^\pm =|p\Q^{\mp}(2n-3,q)|$, so we may assume that $s^\pm> 0$. Using (II) this implies that $\frac q2\leq s^{\pm}$. Moreover, since $t=1$, $s^\pm\leq q$. Plugging in $\frac q2\leq s^\pm \leq q$ in \eqref{x-s} yields that $|\Q(2n-2,q)|\leq x^+\leq |p\Q^{+}(2n-3,q)|$ and $|p\Q^-(2n-3,q)|\leq x^-\leq |\Q(2n-2,q)|$.
 We now show that $x^\pm=|\Q(2n-2,q)|$ or $x^\pm=|p\Q^{\pm}(2n-3,q)|$. 

Consider a fixed hyperplane $\Sigma\in\H^\pm$.
We will first count pairs $(P,\pi)$ where $P$ is a black point in a codimension $2$ space $\pi$ belonging to $\Sigma$ and then triples $(P,P',\pi)$ with $P,P'$ black points, $\pi\in\Sigma$ where $P\in \pi$, $P'\in \pi$, $P\neq P'$. If $\pi_i$ is a codimension $2$ space belonging to $\Sigma$, let $x_i^\pm$ denote the number of black points in $\pi_i$.
We have
$$ \sum_i x_i^\pm  = |\Q^\pm(2n-1,q)|\frac{q^{2n-2}-1}{q-1}$$
$$\sum_i x_i^\pm(x_i^\pm-1) = |\Q^\pm(2n-1,q)|(|\Q^\pm(2n-1,q)|-1)\frac{q^{2n-3}-1}{q-1}$$

So we obtain that
$$ \sum_i (x_i^\pm-|\Q(2n-2,q)|)(x_i^\pm-|p\Q^{\pm}(2n-3,q)|) =  0.$$
As
$|\Q(2n-2,q)|\leq x_i^+\leq |p\Q^{+}(2n-3,q)|$ and $|p\Q^-(2n-3,q)|\leq x_i^-\leq |\Q(2n-2,q)|$, we 
have $x_i^\pm=|\Q(2n-2,q)|$ or $|p\Q^{\pm}(2n-3,q)|$ as required. 
\end{proof}

 {\bfseries Proof of Theorem~\ref{th:main}}
 Let $\mathcal B$ be the set of black points. Since $\H^\pm$ satisfies Condition (I), by Proposition \ref{pr:main}, $\mathcal B$ is a quasi-quadric or $n=2$ and $\H^-$ is the set of solids disjoint from a hyperoval $\O$.
In the latter case, Lemma \ref{axiom-verification} shows that $\H^-$ satisfies Conditions (I) and (II) so we are done. In the former case, Lemma \ref{plane-count-new} shows that we can invoke Theorem \ref{SDW-JS} to conclude that
 $\mathcal B$ is the set of points of a non-singular quadric of $\PG(2n,q)$. Furthermore, by Proposition \ref{pr:main} $\H^\pm$ is the set of hyperplanes of $\PG(2n,q)$ that meet a non-singular quadric in a $\Q^\pm(2n-1,q)$ as required. We note that it follows that $\M^\pm$ is the set of solids that meet the non-singular quadric in a $\Q^\mp(2n-1,q)$, and $\T$ is the set of solids that meet the non-singular quadric in a quadratic cone. \hfill $\square$

\section{Conclusion}

Many variations on this theme are possible: one can consider subspaces of different dimensions for different types of quadrics. Characterisations capturing a broad class of geometries and/or with 
unexpected geometries in the conclusion are of particular interest.

A concrete open problem is to extend the characterisation of Theorem \ref{th:main} to odd characteristic. In this case, four intersection numbers would appear in Condition (II)  as can be seen from the following lemma.
\begin{lemma}\label{four-options}Let $q$ be odd and let $\H^\pm$ be the set of hyperplanes meeting $\Q=\Q(2n,q)$ in a $\Q^\pm(2n-1,q)$. Then every codimension $2$ space is contained in $0, \frac{q-1}{2}, \frac{q+1}{2}$ or $q$ elements of $\H^{\pm}$.
\end{lemma}
\begin{proof}
A codimension $2$ space $\pi$ intersects $\Q=\Q(2n,q)$ either in a cone $L\Q(2n-4,q)$, where $L$ is a line, a cone $P\Q^\pm(2n-1,q)$ where $P$ is a point, or a non-singular parabolic quadric $\Q(2n-2,q)$. A cone $L\Q(2n-4,q)$ lies on $0$ elements of $\H^\pm$ as it only lies on singular hyperplanes. A cone $P\Q^\pm(2n-1,q)$  clearly lies on $1$ tangent hyperplane, and $q$ elements of $\H^\pm$. 
Now assume that $\pi$ intersects $\Q$ in a $\Q(2n-2,q)$. The line $\pi^\perp$, where $\perp$ denotes the polarity associated with $\Q$ intersects $\Q$ in $0$ or $2$ points. If $\pi^\perp$ meets $\Q$ in $0$ points, none of the $q+1$ hyperplanes through $\pi$ is singular. If $s$ is the number of elements of $\H^\pm$ through $\pi$, then 
$$s(|\Q^\pm(2n-1,q)|-|\Q(2n-2,q)|)+(q+1-s)(|\Q^\mp(2n-1,q)|-|\Q(2n-2,q)|)=|\Q(2n,q)|-|\Q(2n-2,q)|,$$ which yields $s=\frac{q+1}{2}$. Similarly, if $\pi^\perp$ meets $\Q$ in $2$ points, then there are precisely $2$ singular hyperplanes through $\pi$ and we have that 
\[s(|\Q^\pm(2n-1,q)|-|\Q(2n-2,q)|)+(q-1-s)(|\Q^\mp(2n-1,q)|-|\Q(2n-2,q)|)=\]\[|\Q(2n,q)|-|\Q(2n-2,q)|-2(|P\Q(2n-2,q)|-|\Q(2n-2,q)|)\] which implies that $s=\frac{q-1}{2}$.
The lemma follows.
\end{proof}

Whilst we have not pursued this in great depth, there seems to be a tight balance between having enough information to start deriving results caused by the fact there are four possibilities in Lemma \ref{four-options} and adding in a hypothesis which renders the problem trivial. Even the analogon of Proposition \ref{pr:main} seems harder to prove when $q$ is odd as the elementary number theory techniques we used for $q$ even don't work when $q$ is odd.

{\bf Address of the authors:}\\

Jeroen Schillewaert\\
Department of Mathematics\\
University of Auckland\\
38 Princes Street, Auckland Central, Auckland 1010\\
New Zealand

Geertrui Van de Voorde\\
School of Mathematics and Statistics\\
University of Canterbury \\
Private bag 4800\\
8140 Christchurch\\
New Zealand
\end{document}